\documentclass[a4paper,11pt,reqno]{amsart}
\usepackage{mathrsfs}
 \usepackage[all]{xy}
\usepackage{amsmath,amssymb,amscd,bbm,amsthm,mathrsfs}
\usepackage{color,xcolor}
\usepackage{graphicx}
\usepackage{manfnt}
\newtheorem{thm}{Theorem}[section]

\newtheorem{cor}{Corollary}[section]
\newtheorem{con}{Conjecture}[section]

\newtheorem{rem}{Remark}[section]
\newtheorem{exam}{Example}[section]

\begin{document}
\numberwithin{equation}{section}

 \title[Hermitian manifolds with nonpositive holomorphic sectional curvature]{Hermitian manifolds with nonpositive holomorphic sectional curvature}
\author{Kai Tang}
\address{Kai Tang. School of Mathematical Sciences, Zhejiang Normal University, Jinhua, Zhejiang, 321004, China} \email{{kaitang001@zjnu.edu.cn}}
\thanks{\text{Foundation item:} Supported by Natural Science Foundation of Zhejiang Province (No. LMS26A010006).}
\begin{abstract}
We study compact K\"ahler manifolds admitting Hermitian metrics with nonpositive holomorphic sectional
curvature. We prove that the canonical bundle of such a manifold is
nef, removing the pluriclosed assumption from the corresponding
nefness result of Broder-Stanfield \cite{BroderStanfield}. In complex dimension two, we
further show that negative  holomorphic sectional curvature
implies the ampleness of the canonical bundle. In complex dimension three, the same conclusion holds if the Hermitian metric is additionally assumed to be balanced. We also prove that vanishing  holomorphic sectional
curvature forces the first Chern class to vanish.
\end{abstract}

 \maketitle

%%\tableofcontents
\section{Introduction}
%%\subsection{Background}
The relationship between curvature and the positivity of canonical
bundles is a central topic in complex differential geometry. In the
K\"ahler setting, Wu-Yau \cite{WuYau} proved that a projective manifold admitting
a K\"ahler metric with negative holomorphic sectional curvature has
ample canonical bundle. Tosatti-Yang \cite{TosattiYang} subsequently
removed the projectivity assumption and proved that a compact K\"ahler
manifold with nonpositive holomorphic sectional curvature has nef
canonical bundle. In the strictly negative case, the canonical bundle is ample. The quasi-negative case was further studied by Diverio-Trapani \cite{DiverioTrapani} and Wu-Yau
\cite{WuYauRemark}. Nomura later gave an alternative
proof of these results using the K\"ahler-Ricci flow \cite{Nomura}. For more related works, we refer readers to
\cite{CLT,LNZ,Tang1,Tang2,ZhangSY,ZhangYS}.

The corresponding questions for arbitrary Hermitian metrics remain
largely open. Yang-Zheng \cite{YangZheng} proposed the following conjecture:
\begin{con}
\begin{itemize}
\item[(1)] if $X$ is Kobayashi hyperbolic, then $K_X$ is ample;
\item[(2)] if $X$ admits a Hermitian metric with quasi-negative holomorphic sectional curvature, then $K_X$ is ample;
\item[(3)] if $X$ admits a Hermitian metric with negative holomorphic sectional curvature, then $K_X$ is ample.
\end{itemize}
\end{con}
Here quasi-negative means nonpositive everywhere and strictly negative
at some point. By Yau's Schwarz lemma \cite{YauSchwarz}, the first conjecture implies the
third one, while the second also contains the third as a special case.
Thus, extending the Wu-Yau theory from K\"ahler metrics to general
Hermitian metrics is part of a natural program relating
curvature, hyperbolicity, and positivity of the canonical bundle.

To study these questions, Yang-Zheng \cite{YangZheng} introduced the real
bisectional curvature of a Hermitian metric. Its sign
is equivalent to that of the holomorphic sectional curvature in the
K\"ahler case, but it is a stronger condition for a general Hermitian
metric. This stronger curvature notion is adapted to the Hermitian
Schwarz lemma. They proved that if a compact K\"ahler manifold
admits a Hermitian metric with nonpositive real bisectional curvature,
then its canonical bundle is nef; if the real bisectional curvature is
quasi-negative, then the canonical bundle is ample. Here the given
curvature metric need not itself be K\"ahler.

Further progress has been obtained under additional geometric
assumptions. Broder-Stanfield \cite{BroderStanfield} proved a general Schwarz lemma for
Hermitian manifolds and showed that if a compact K\"ahler manifold
admits a pluriclosed Hermitian metric with nonpositive holomorphic sectional curvature, then $K_X$ is nef; strict negativity
implies that $K_X$ is ample. In complex
dimension two, Lee-Streets \cite{LeeStreets} proved, by means of the pluriclosed
flow, that negative real bisectional curvature implies the ampleness
of the canonical bundle without assuming that the underlying complex
surface is K\"ahler.

For an arbitrary Hermitian metric, however, holomorphic sectional
curvature alone does not provide the curvature estimate required by
the known Schwarz lemma arguments. Consequently, the methods based on
complex Monge-Amp\`ere equations or geometric flows do not directly
apply. It is natural to ask whether a compact K\"ahler
manifold admitting a Hermitian metric with nonpositive holomorphic sectional curvature must have nef canonical bundle. In this paper, we give an affirmative answer to this question. Our
result replaces the nonpositive real bisectional curvature assumption
in the theorem of Yang-Zheng \cite{YangZheng} by the weaker condition of nonpositive holomorphic sectional curvature. It also removes the
pluriclosed assumption from the corresponding nefness result of
Broder-Stanfield \cite{BroderStanfield}. Our main result is the following.

\begin{thm}\label{thm:main}
Let $X$ be a compact K\"ahler manifold and let $h$ be a Hermitian
metric on $X$.
\begin{itemize}
    \item[(1)] If the holomorphic sectional curvature of $h$ is
    nonpositive, then the canonical bundle $K_X$ is nef.

    \item[(2)] If the holomorphic sectional curvature of $h$ vanishes, then $c_1(X)=0$. Consequently, $X$ admits a Ricci-flat K\"ahler metric.
\end{itemize}
\end{thm}
The proof of Theorem~\ref{thm:main}(1) is based on recent results of Ou \cite{Ou} and Cao-H\"oring \cite{CaoHoring} in
the birational geometry of compact K\"ahler manifolds. The curvature
assumption excludes rational curves on $X$. Ou's characterization of
uniruled compact K\"ahler manifolds then implies that $K_X$ is
pseudoeffective, while the results of Cao-H\"oring, together with
Ou's theorem, show that the non-nefness of $K_X$ would produce a
rational curve. This gives a contradiction.

The vanishing case has also been studied by Broder-Tang
\cite{BroderTang}. They proved that if a compact K\"ahler manifold
admits a pluriclosed Hermitian metric with  vanishing holomorphic sectional curvature, then the metric itself is
K\"ahler and hence flat. Theorem~\ref{thm:main}(2) removes the
pluriclosed assumption, but yields the weaker conclusion that
$c_1(X)=0$. The Calabi-Yau theorem \cite{YauCalabi} then provides a possibly
different Ricci-flat K\"ahler metric.

As an application of Theorem~\ref{thm:main}(1), we obtain the following
surface consequence.

\begin{cor}\label{cor1.1}
Let $X$ be a compact K\"ahler surface admitting a Hermitian metric
with negative holomorphic sectional curvature. Then $K_X$ is
ample.
\end{cor}
The analogous statement for K\"ahler metrics with negative
holomorphic sectional curvature is classical in complex dimension
two (see \cite{Wong1981,HeierLuWong}). The distinction here is that the
Hermitian metric carrying the negative  holomorphic sectional
curvature is not assumed to be K\"ahler. Our proof combines the
nefness obtained in Theorem~\ref{thm:main}(1) with standard results
from the theory of compact K\"ahler surfaces.

Under the balanced assumption, the surface result also extends
to complex dimension three.

\begin{cor}\label{cor:threefold}
Let $X$ be a compact K\"ahler threefold admitting a balanced Hermitian
metric with negative holomorphic sectional curvature. Then $K_X$ is ample.
\end{cor}

The proof combines Theorem 1.1(1) with the recent log abundance theorem
for compact K\"{a}hler threefolds \cite{DasOu1,DasOu2}. The balanced condition rules out the Kodaira dimension zero case, while the remaining intermediate Kodaira dimensions are excluded by restricting the metric to general fibers of the Iitaka fibration.

We next turn to the vanishing case. The following theorem provides a partial converse to Theorem~\ref{thm:main}(2) under the additional assumption that the Hermitian metric have nonnegative holomorphic sectional curvature. However, it does not require any K\"ahler assumption.

\begin{thm}\label{thm1.2}
Let $X$ be a compact complex manifold with $c_1^{BC}(X)=0$. Then every Hermitian metric $h$ with nonnegative holomorphic sectional curvature has identically vanishing holomorphic sectional curvature.
\end{thm}

Combining the two theorems with the
$\partial\bar\partial$-lemma, we obtain the following characterization
in the compact K\"ahler setting.
\begin{cor}\label{cor1.2}
Let $X$ be a compact K\"ahler manifold and let $h$ be a Hermitian
metric on $X$. Then
\[
H_h\equiv0
\quad\Longleftrightarrow\quad
c_1(X)=0
\ \text{and}\ H_h\geq0.
\]
\end{cor}

\begin{rem} The nonnegativity assumption in Corollary~\ref{cor1.2} is essential.
Indeed, as shown in Example~\ref{ex:K3}, there exists a compact
K\"ahler surface with $c_1(X)=0$ which admits no Hermitian metric
with vanishing holomorphic sectional curvature.
\end{rem}

This paper is organized as follows. In section 2, we provide some basic knowledge which will be used in our proofs. In section 3, we prove the main results.

\section{Preliminaries}
\subsection{Curvatures of Hermitian metrics}
Let $X$ be a compact complex manifold of
complex dimension $n$. Let $g$ be a Hermitian metric on $X$. In local
holomorphic coordinates $(z^1,\ldots,z^n)$, we write
\[
g=g_{i\bar j}\,dz^i\otimes d\bar z^j,
\]
and denote by
\[
\omega=\sqrt{-1}\,g_{i\bar j}\,dz^i\wedge d\bar z^j
\]
its associated fundamental $(1,1)$-form. We shall often identify $g$
with $\omega$ when no confusion can arise. The Hermitian metric $g$
is called K\"ahler if $d\omega=0$. Let $\nabla$ be the Chern connection of $(T^{1,0}X,g)$. Its curvature
tensor is locally given by
\[
R_{i\bar j k\bar\ell}
=
-\frac{\partial^2g_{k\bar\ell}}
{\partial z^i\partial\bar z^j}
+
g^{p\bar q}
\frac{\partial g_{k\bar q}}{\partial z^i}
\frac{\partial g_{p\bar\ell}}{\partial\bar z^j}.
\]
For a nonzero vector
$\xi=\xi^i\frac{\partial}{\partial z^i}\in T^{1,0}_xX$, the holomorphic sectional curvature is defined by
\[
H_g(\xi)
=
\frac{
R_{i\bar j k\bar\ell}
\xi^i\overline{\xi^j}\xi^k\overline{\xi^\ell}
}{
|\xi|_g^4
}.
\]

On a general Hermitian manifold, there are four natural Chern--Ricci
curvatures:
\[
\begin{aligned}
R^{(1)}_{i\bar j}
&=g^{k\bar\ell}R_{i\bar j k\bar\ell},&
R^{(2)}_{i\bar j}
&=g^{k\bar\ell}R_{k\bar\ell i\bar j},\\
R^{(3)}_{i\bar j}
&=g^{k\bar\ell}R_{i\bar\ell k\bar j},&
R^{(4)}_{i\bar j}
&=g^{k\bar\ell}R_{k\bar j i\bar\ell}.
\end{aligned}
\]
The first Chern-Ricci form satisfies
\[
Ric^{(1)}(\omega)
=
\sqrt{-1}\,R^{(1)}_{i\bar j}\,dz^i\wedge d\bar z^j
=
-\sqrt{-1}\,\partial\bar\partial\log\det(g_{i\bar j}),
\]
which is a $(1,1)$-form representing the first Chern class $c_{1}(X)$.
The two Chern scalar curvatures are
\begin{align}
s_g
=g^{i\bar j}g^{k\bar\ell}R_{i\bar j k\bar\ell},\,\,\,\,\,\,\,
\widehat s_g=g^{i\bar\ell}g^{k\bar j}R_{i\bar j k\bar\ell}.
\end{align}
Equivalently, $s_g$ is the trace of either $R^{(1)}$ or $R^{(2)}$,
whereas $\widehat s_g$ is the trace of either $R^{(3)}$ or
$R^{(4)}$. The total scalar curvature is
\begin{align}
\int_{X}s_{g}\omega^{n}=n\int_{X}Ric(\omega)\wedge \omega^{n-1}.
\end{align}

The Berger averaging formula \cite{Berger1966} gives
\[
\int_{\{\xi\in T^{1,0}_xX:\,|\xi|_g=1\}}
H_g(\xi)\,d\theta(\xi)
=
\frac{s_g+\widehat s_g}{n(n+1)}.
\]
In particular,
\[
H_g\geq0
\quad\Longrightarrow\quad
s_g+\widehat s_g\geq0.
\]

If $g$ is K\"ahler, then the Chern connection agrees with the
Levi-Civita connection and
\[
R_{i\bar j k\bar\ell}
=
R_{k\bar j i\bar\ell}
=
R_{i\bar\ell k\bar j}.
\]
Consequently,
\[
R^{(1)}=R^{(2)}=R^{(3)}=R^{(4)}
\qquad\text{and}\qquad
s_g=\widehat s_g.
\]
Moreover, the K\"ahler curvature tensor is determined by its
holomorphic sectional curvature.

\subsection{Special Hermitian metrics}

A Hermitian metric $\omega$ is called Gauduchon if
\[
\partial\bar\partial\omega^{n-1}=0,
\]
balanced if
\[
d\omega^{n-1}=0,
\]
and pluriclosed if
\[
\partial\bar\partial\omega=0.
\]
It is proved by Gauduchon \cite{Gauduchon1977}, that, in the conformal class of each Hermitian metric, there exists a unique Gauduchon metric (up to a positive constant scaling).

For a Hermitian metric $\omega$, one has
\[
\int_X(s_\omega-\widehat s_\omega)\,\omega^n
=
\int_X|\partial^*\omega|_\omega^2\,\omega^n.
\]
In particular, if $\omega$ is balanced, then
$\partial^*\omega=0$, and hence
\[
\int_X (s_\omega-\widehat{s}_\omega)\omega^n=0.
\].

\subsection{Line bundles and Bott-Chern classes}

Let $L\to X$ be a holomorphic line bundle. If $h$ is a smooth
Hermitian metric on $L$, its Chern curvature is denoted by
$\Theta_h(L)$. The first Chern class is
\[
c_1(L)
=
\left[
\frac{\sqrt{-1}}{2\pi}\Theta_h(L)
\right]
\in H^2(X,\mathbb R).
\]
The canonical bundle is
\[
K_X=\Lambda^n(T^{1,0}X)^*,
\]
and
\[
c_1(X)=c_1(K_X^{-1})=-c_1(K_X).
\]
Fix a K\"ahler form
$\omega_0$. A holomorphic line bundle $L$ is called nef if, for every
$\varepsilon>0$, there exists a smooth Hermitian metric
$h_\varepsilon$ on $L$ such that
\[
\frac{\sqrt{-1}}{2\pi}\Theta_{h_\varepsilon}(L)
\geq-\varepsilon\omega_0.
\]
It is called pseudoeffective if $c_1(L)$ contains a closed positive
$(1,1)$-current. A line bundle $L$ is called big if $\kappa(L)=\dim X$.
Equivalently, the class $c_1(L)$ contains a K\"ahler current.
For K\"{a}hler manifolds, if $L$ is nef, then $L$ is big if and only if
\[
\int_Xc_1(L)^n>0.
\]
It is called ample if some positive tensor
power defines a holomorphic embedding into projective space;
equivalently, $L$ admits a smooth Hermitian metric with positive
curvature.

The Kodaira dimension of $L$ is defined by
\[
\kappa(L)
=
\begin{cases}
-\infty,
&
H^0(X,L^{\otimes m})=0
\text{ for every }m\geq1,\\[4pt]
\displaystyle
\limsup_{m\to\infty}
\frac{\log h^0(X,L^{\otimes m})}{\log m},
&
\text{otherwise}.
\end{cases}
\]
The Kodaira dimension of $X$ is
\[
\kappa(X)=\kappa(K_X).
\]

We also recall some standard facts concerning semiample line bundles.
A holomorphic line bundle $L$ on a compact complex manifold $X$ is
called \emph{semiample} if $L^m$ is globally generated for some
positive integer $m$. If $L$ is semiample, then, for $m>0$ sufficiently
divisible, it induces a holomorphic map
\[
\varphi:X\longrightarrow Y
\]
with connected fibers such that
\[
L^m\simeq\varphi^*A
\]
for some ample line bundle $A$ on $Y$, and
\[
\dim Y=\kappa(X,L).
\]
In particular, if $\kappa(X,L)=0$, then
\[
L^m\simeq\mathcal O_X
\]
for some $m>0$. When $L=K_X$, we refer to $\varphi$ as the Iitaka
fibration.

The real Bott-Chern cohomology group is
\[
H^{1,1}_{BC}(X,\mathbb R)
=
\frac{
\{\alpha\in A^{1,1}(X,\mathbb R):d\alpha=0\}
}{
\{\sqrt{-1}\partial\bar\partial\varphi:
\varphi\in C^\infty(X,\mathbb R)\}
}.
\]
The first Bott-Chern class of $L$ is
\[
c_1^{BC}(L)
=
\left[
\frac{\sqrt{-1}}{2\pi}\Theta_h(L)
\right]_{BC},
\]
which is independent of the choice of $h$. We set
\[
c_1^{BC}(X)=c_1^{BC}(K_X^{-1}).
\]
In particular,
\[
\frac{1}{2\pi}Ric^{(1)}(\omega)
\in c_1^{BC}(X).
\]

There is a natural map
\[
H^{1,1}_{BC}(X,\mathbb R)
\longrightarrow H^2(X,\mathbb R),
\]
and hence
\[
c_1^{BC}(X)=0
\quad\Longrightarrow\quad
c_1(X)=0.
\]
If $X$ is  a compact K\"ahler manifold, the $\partial\bar\partial$-lemma implies
the converse, so that
\[
c_1^{BC}(X)=0
\quad\Longleftrightarrow\quad
c_1(X)=0.
\]

Finally, a rational curve on $X$ is the image of a nonconstant
holomorphic map from $\mathbb P^1$ to $X$, and $X$ is called
uniruled if it is covered by rational curves. $X$ is said to be Brody hyperbolic if every
holomorphic map $\mathbb{C}\rightarrow X$ is constant. The celebrated Brody criterion \cite{Brody,Kobayashi1998} asserts that $X$ is Kobayashi hyperbolic if and only if it is Brody hyperbolic.

\section{Proofs of the main results}
In this section, we present the proofs of Theorem~\ref{thm:main},
Corollary~\ref{cor1.1}, Corollary~\ref{cor:threefold} and Theorem~\ref{thm1.2}.

\begin{proof}[Proof of Theorem~\ref{thm:main}]

\noindent\emph{Proof of \emph{(1)}.}
We first show that $X$ contains no rational curves. This fact is
essentially a rank-one consequence of Schwarz lemma arguments of
Yau-Royden type \cite{YauSchwarz,Royden}; for completeness, we give the local Chern--Lu
calculation.
We assume that there exists a nonconstant holomorphic
map
\[
f:\mathbb P^1\longrightarrow X.
\]
Choose a local holomorphic coordinate
$z^1$ on $\mathbb P^1$ and local holomorphic coordinates
$(w^1,\ldots,w^n)$ on $X$. Let $g$ be a Fubini--Study metric on $\mathbb P^1$, normalized so that
its Ricci curvature satisfies
\[
R^g_{1\bar1}=\tau g_{1\bar1}
\]
for some constant $\tau>0$.  Write
\[
f^\alpha=w^\alpha\circ f,
\qquad
f^\alpha_1=\frac{\partial f^\alpha}{\partial z^1},
\]
and define
\[
u=\operatorname{tr}_{g}f^*h
  =g^{1\bar1}h_{\alpha\bar\beta}
    f^\alpha_1\overline{f^\beta_1}.
\]
Clearly, $u\geq0$, and $u\equiv0$ if and only if $f$ is constant.

The Chern-Lu formula gives
\begin{align}
\Delta_g u
=|\nabla df|^2
+(g^{1\bar1})^2R^g_{1\bar1}
 h_{\alpha\bar\beta}
 f^\alpha_1\overline{f^\beta_1}-(g^{1\bar1})^2
R^h_{\alpha\bar\beta\gamma\bar\delta}
f^\alpha_1\overline{f^\beta_1}
f^\gamma_1\overline{f^\delta_1}.
\label{eq:Chern-Lu-P1}
\end{align}
The source curvature term is
\[
(g^{1\bar1})^2R^g_{1\bar1}
h_{\alpha\bar\beta}
f^\alpha_1\overline{f^\beta_1}
=
\tau u.
\]
At a point where $df\neq0$, set
\[
\xi
=
f^\alpha_1\frac{\partial}{\partial w^\alpha}
\in T^{1,0}X.
\]
Then
\begin{align*}
(g^{1\bar1})^2
R^h_{\alpha\bar\beta\gamma\bar\delta}
f^\alpha_1\overline{f^\beta_1}
f^\gamma_1\overline{f^\delta_1}=(g^{1\bar1})^2
R^h(\xi,\bar\xi,\xi,\bar\xi)
=H_h(\xi)
 \bigl(g^{1\bar1}|\xi|_h^2\bigr)^2
=H_h(\xi)u^2.
\end{align*}
Since $H_h\leq0$, the last curvature contraction in
\eqref{eq:Chern-Lu-P1} is nonpositive. Consequently,
\[
\Delta_g u
=
|\nabla df|^2+\tau u-H_h(\xi)u^2
\geq \tau u.
\]
Integrating over $\mathbb P^1$, we obtain
\[
0
=
\int_{\mathbb P^1}\Delta_g u\,dV_g
\geq
\tau\int_{\mathbb P^1}u\,dV_g.
\]
Hence $u\equiv0$, contradicting the nonconstancy of $f$. Thus $X$
contains no rational curves.

We now apply two results from the birational geometry of compact
K\"ahler manifolds. Ou \cite{Ou} proved that a compact K\"ahler manifold is
uniruled if and only if its canonical bundle is not
pseudoeffective. Since $X$ contains no rational curves, it
is not uniruled, and therefore $K_X$ is pseudoeffective.

Suppose that $K_X$ were not nef. Cao-H\"oring \cite[Theorem~1.3]{CaoHoring} proved that, for an
$n$-dimensional compact K\"ahler manifold whose canonical bundle is
pseudoeffective but not nef, there exists a $K_X$-negative rational
curve, provided that the equivalence
\[
K_Y\ \text{is pseudoeffective}
\quad\Longleftrightarrow\quad
Y\ \text{is not uniruled}
\]
holds for compact K\"ahler manifolds $Y$ of dimension at most
$n-1$. Since Ou's theorem establishes
this equivalence in every dimension, the hypothesis of the
Cao-H\"oring theorem is satisfied. Hence there exists a nonconstant
holomorphic map
\[
\varphi:\mathbb P^1\longrightarrow X
\]
such that
\[
K_X\cdot\varphi(\mathbb P^1)<0.
\]
This contradicts the absence of rational curves proved above.
Therefore, $K_X$ is nef.

\medskip
\noindent\emph{Proof of \emph{(2)}.}
Assume now that
\[
H_h\equiv0,
\]
and let $\omega$ be the fundamental form associated with $h$. Let
\[
\omega_G=f^{\frac{1}{n-1}}\omega,
\qquad f>0,
\]
be the Gauduchon metric in the conformal class of $\omega$. By the
Berger averaging formula,
\[
s_h+\widehat{s}_h=0.
\]
The conformal Gauduchon identity
\cite[Equation~(3.8)]{Yang2016} gives
\begin{equation}\label{eq:Gauduchon-weighted}
\int_Xs_G\,\omega_G^n
=
\frac12\int_X
f\bigl(s_h+\widehat{s}_h\bigr)\omega^n
+
\frac12\int_X
\left|\partial_G^*\omega_G\right|_{\omega_G}^2
\omega_G^n.
\end{equation}
It follows that
\begin{equation}\label{eq:total-scalar-nonnegative}
\int_Xs_G\,\omega_G^n
=
\frac12\int_X
\left|\partial_G^*\omega_G\right|_{\omega_G}^2
\omega_G^n
\geq0.
\end{equation}

By part \emph{(1)}, $K_X$ is nef. Fix a K\"ahler form $\omega_0$ on
$X$. For every $\varepsilon>0$, there exists a smooth Hermitian metric
$k_\varepsilon$ on $K_X$ whose normalized curvature form
\[
\alpha_\varepsilon
=
\frac{\sqrt{-1}}{2\pi}
\Theta_{k_\varepsilon}(K_X)
\]
satisfies
\[
\alpha_\varepsilon\geq-\varepsilon\omega_0.
\]
Both $\alpha_\varepsilon$ and $-\frac{1}{2\pi}\operatorname{Ric}^{(1)}(\omega_G)$ represent $c_1(K_X)$. Since $\omega_G$ is Gauduchon metric,
\begin{align}
\int_X\alpha_\varepsilon\wedge\omega_G^{n-1}=
-\frac{1}{2\pi}
\int_X\operatorname{Ric}^{(1)}(\omega_G)
       \wedge\omega_G^{n-1}=-\frac{1}{2\pi n}
\int_Xs_G\,\omega_G^n.
\label{eq:nef-pairing}
\end{align}
The inequality
$\alpha_\varepsilon\geq-\varepsilon\omega_0$ therefore implies
\[
-\frac{1}{2\pi n}
\int_Xs_G\,\omega_G^n
\geq
-\varepsilon
\int_X\omega_0\wedge\omega_G^{n-1}.
\]
Equivalently,
\[
\int_Xs_G\,\omega_G^n
\leq
2\pi n\varepsilon
\int_X\omega_0\wedge\omega_G^{n-1}.
\]
Letting $\varepsilon\to0$, we obtain
\[
\int_Xs_G\,\omega_G^n\leq0.
\]
Combining this with \eqref{eq:total-scalar-nonnegative}, we conclude
that
\[
\int_Xs_G\,\omega_G^n=0.
\]

It remains to show that the first Chern class vanishes. From
\eqref{eq:nef-pairing}, we now have
\[
\int_X\alpha_\varepsilon\wedge\omega_G^{n-1}=0.
\]
Set
\[
\beta_\varepsilon
=
\alpha_\varepsilon+\varepsilon\omega_0.
\]
Then $\beta_\varepsilon$ is a smooth closed semipositive $(1,1)$-form,
and
\[
\int_X\beta_\varepsilon\wedge\omega_G^{n-1}
=
\varepsilon
\int_X\omega_0\wedge\omega_G^{n-1}
\longrightarrow0.
\]
Since $\omega_G^{n-1}$ is strictly positive, the mass of the positive
current $\beta_\varepsilon$ tends to zero. Therefore,
\[
\beta_\varepsilon\longrightarrow0
\]
in the sense of currents. Hence
\[
\alpha_\varepsilon
=
\beta_\varepsilon-\varepsilon\omega_0
\longrightarrow0
\]
in the sense of currents.

On the other hand, every $\alpha_\varepsilon$ represents the fixed
de Rham class $c_1(K_X)$. Thus, for every smooth closed real
$(2n-2)$-form $\Psi$,
\[
\left\langle c_1(K_X),[\Psi]\right\rangle
=
\int_X\alpha_\varepsilon\wedge\Psi
\longrightarrow0.
\]
By Poincar\'e duality,
\[
c_1(K_X)=0\in H^2(X,\mathbb R).
\]
Consequently,
\[
c_1(X)=-c_1(K_X)=0.
\]
The Calabi-Yau theorem \cite{YauCalabi} then yields a Ricci-flat K\"ahler metric on
$X$.
\end{proof}

\begin{proof}[Proof of Corollary~\ref{cor1.1}]
By Theorem~\ref{thm:main}(1), the canonical bundle $K_X$ is nef.
Moreover, by the Schwarz lemma \cite{YauSchwarz}, the negativity of the holomorphic sectional curvature implies that $X$ is Kobayashi hyperbolic.

We first show that $X$ is of general type. By the
Enriques--Kodaira classification \cite[Chapter~VI]{BHPV},
the case $\kappa(X)=-\infty$ would imply that $X$ is rational or
ruled, and hence contains a rational curve, contradicting
hyperbolicity. If $\kappa(X)=1$, then the Iitaka fibration has a
general fiber which is an elliptic curve, again a contradiction. If
$\kappa(X)=0$, then a finite \'etale cover of $X$ is either a complex
torus or a K3 surface. Neither is Kobayashi hyperbolic; for the K3
case, see \cite{KLV}. Hence
\[
\kappa(X)=2.
\]
Thus $K_X$ is big. Since it is also nef,
\[
K_X^2>0.
\]
Moreover, the bigness of $K_X$ implies that $X$ is Moishezon, and
hence projective since $X$ is K\"ahler.

Let $C\subset X$ be an irreducible curve. By nefness,
\[
K_X\cdot C\geq0.
\]
Suppose that $K_X\cdot C=0$. Since $K_X^2>0$ and $K_X\cdot C=0$, the Hodge index theorem shows
that $C^2<0$, as the intersection form is negative definite on the
orthogonal complement of $K_X$.
By the adjunction formula,
\[
2p_a(C)-2
=
(K_X+C)\cdot C
=
C^2<0.
\]
It follows that
\[
p_a(C)=0
\qquad\text{and}\qquad
C^2=-2.
\]
Since $p_a(C)=0$, the normalization of $C$ is  $\mathbb{P}^{1}$, so $C$ is a rational curve, contradicting the Kobayashi hyperbolicity. Therefore,
\[
K_X\cdot C>0
\]
for every irreducible curve $C\subset X$.

Since $K_X^2>0$, the Nakai-Moishezon criterion
\cite{Lazarsfeld} implies that $K_X$ is ample.
\end{proof}

\begin{proof}[Proof of Corollary~\ref{cor:threefold}]
By Theorem~1.1(1), the canonical bundle $K_X$ is nef. By the abundance
theorem for compact K\"ahler threefolds
\cite{DasOu1,DasOu2}, $K_X$ is semiample. We first show that
\[
\kappa(X)=3.
\]

Suppose first that $\kappa(X)=0$. Since $K_X$ is semiample, for some
$m>0$ sufficiently divisible, $|mK_X|$ is base-point free. Since
$\kappa(X)=0$, we have
\[
mK_X\simeq\mathcal O_X.
\]
Hence
\[
c_1^{BC}(X)=0.
\]

Let $\omega$ be the given balanced Hermitian metric. Since $\omega$ is
balanced, it is Gauduchon and
\[
\int_X(s_\omega-\widehat{s}_\omega)\omega^3=0.
\]
Moreover,
\[
\int_Xs_\omega\omega^3
=
3\int_X Ric^{(1)}(\omega)\wedge\omega^2=0,
\]
because $c_1^{BC}(X)=0$ and $d\omega^2=0$. Therefore,
\[
\int_X(s_\omega+\widehat{s}_\omega)\omega^3=0.
\]
On the other hand, the Berger averaging formula and the negativity of
the holomorphic sectional curvature imply
\[
s_\omega+\widehat{s}_\omega<0,
\]
a contradiction. Thus
\[
\kappa(X)\neq0.
\]

Now suppose that $\kappa(X)=1$ or $2$. Since $K_X$ is semiample, for some
sufficiently divisible integer $m>0$, the Iitaka fibration is given by
\[
\varphi:X\longrightarrow Y,
\]
where
\[
mK_X\simeq\varphi^*A
\]
for some ample line bundle $A$ on $Y$, and
\[
\dim Y=\kappa(X).
\]
Let $F$ be a general smooth fiber of $\varphi$. Then
\[
mK_X|_F\simeq\mathcal O_F.
\]
Since $F$ is a fiber of $\varphi$, its normal bundle in $X$ is trivial.
Hence, by the adjunction formula,
\[
K_F\simeq K_X|_F,
\]
and therefore
\[
mK_F\simeq\mathcal O_F.
\]

By the Chern--Gauss equation, the Hermitian metric induced by $\omega$
on $F$ still has negative holomorphic sectional curvature.

If $\kappa(X)=2$, then $F$ is a compact Riemann surface. Since
$mK_F\simeq\mathcal O_F$, we have
\[
\deg K_F=0,
\]
and hence $F$ is an elliptic curve, contradicting the Gauss--Bonnet
theorem.

If $\kappa(X)=1$, then $F$ is a compact K\"ahler surface. By
Corollary~1.1, $K_F$ is ample, contradicting
\[
mK_F\simeq\mathcal O_F.
\]
Therefore
\[
\kappa(X)=3.
\]

Consequently, $K_X$ is nef and big. Hence $X$ is Moishezon, and therefore
projective since $X$ is K\"ahler. Moreover, as in the proof of
Theorem~1.1(1), the negativity of the holomorphic sectional curvature
implies that $X$ contains no rational curves.

Since $K_X$ is nef and big, the base-point-free theorem gives a morphism
\[
\psi:X\longrightarrow Z
\]
for some sufficiently divisible $m>0$ such that
\[
mK_X\simeq\psi^*A,
\]
where $A$ is ample on $Z$. If $\psi$ were not finite, then there would
exist a positive-dimensional fiber. Since $mK_X=\psi^*A$, the canonical
bundle $K_X$ is numerically trivial on the fibers of $\psi$. By
Kawamata's theorem \cite[Theorem~2]{Kawamata91}, such a fiber contains a
rational curve, contradicting the negativity of the holomorphic
sectional curvature. Hence $\psi$ is finite.

It follows that $\psi^*A$ is ample. Therefore $mK_X$, and hence $K_X$,
is ample.
\end{proof}

\begin{proof}[Proof of Theorem~\ref{thm1.2}]
Let $\omega$ be the fundamental form of $h$, and let
\[
\omega_G=f^{\frac{1}{n-1}}\omega,
\qquad f>0,
\]
be the Gauduchon metric in its conformal class. As in the proof of
Theorem~\ref{thm:main}, we have
\[
\int_X s_G\,\omega_G^n
=
\frac12\int_X f\bigl(s_h+\widehat{s}_h\bigr)\omega^n
+
\frac12\int_X
\left|\partial_G^*\omega_G\right|_{\omega_G}^2\omega_G^n.
\tag{3.5}
\]
Since $H_h\geq0$, the Berger averaging formula gives
\[
s_h+\widehat{s}_h\geq0.
\]
Hence the right-hand side of \textup{(3.5)} is nonnegative.

On the other hand, the assumption $c_1^{BC}(X)=0$ implies that
there exists a smooth real-valued function $u$ such that
\[
\operatorname{Ric}^{(1)}(\omega_G)
=
\sqrt{-1}\,\partial\bar\partial u.
\]
Since $\omega_G$ is Gauduchon,
\[
\begin{aligned}
\int_X s_G\,\omega_G^n
&=
n\int_X
\operatorname{Ric}^{(1)}(\omega_G)\wedge\omega_G^{n-1}  \\
&=
n\int_X
\sqrt{-1}\,\partial\bar\partial u\wedge\omega_G^{n-1} \\
&=
n\int_X
u\sqrt{-1}\,\partial\bar\partial\omega_G^{n-1}
=0.
\end{aligned}
\]
It follows from \textup{(3.5)} that
\[
\int_X f\bigl(s_h+\widehat{s}_h\bigr)\omega^n=0.
\]
Since $f>0$ and $s_h+\widehat{s}_h\geq0$, we obtain
\[
s_h+\widehat{s}_h\equiv0.
\]
Applying the Berger averaging formula pointwise, and using
$H_h\geq0$, we conclude that
\[
H_h\equiv0.
\]
\end{proof}

We conclude with a classical example showing that the nonnegativity
assumption in Corollary~\ref{cor1.2} is essential: the condition
$c_1(X)=0$ alone does not guarantee the existence of a Hermitian
metric with vanishing holomorphic sectional
curvature.
\begin{exam}\label{ex:K3}
Consider the Fermat quartic  K3 surface (see \cite{BHPV})
\[
X=\left\{
[z_0:z_1:z_2:z_3]\in\mathbb P^3
\,\middle|\,
z_0^4+z_1^4+z_2^4+z_3^4=0
\right\}.
\]
The surface $X$  is a smooth projective surface in $\mathbb P^3$;
in particular, it is compact. and the adjunction formula gives
\[
K_X
\simeq
\left(K_{\mathbb P^3}\otimes\mathcal O_{\mathbb P^3}(4)\right)|_X
\simeq
\mathcal O_X.
\]
Thus $X$ is a K3 surface and $c_1(X)=0$. Choose $\zeta,\eta\in\mathbb C$ such that
$\zeta^4=\eta^4=-1$. Then the line
\[
L=\{z_0=\zeta z_1,\ z_2=\eta z_3\}\subset\mathbb P^3
\]
is contained in $X$. Hence $X$ contains a rational curve
$L\simeq\mathbb P^1$.

If $X$ admitted a Hermitian metric $h$ with $H_h\equiv0$, then the
argument in the proof of Theorem~\ref{thm:main}(1) would imply that
$X$ contains no rational curves, a contradiction. Therefore, $X$
does not admit any Hermitian metric with vanishing holomorphic sectional curvature.
\end{exam}

\vspace{0.5cm}
\noindent\textbf{Acknowledgement.} The author is grateful to Professor Fangyang Zheng for constant encouragement and support. The author would like to thank OpenAI's ChatGPT-5.5 for helpful
suggestions and assistance in exploring the proof strategy of
Corollary ~\ref{cor:threefold}.
\vspace{0.5cm}
%% \noindent\textbf{Data availability.} Data sharing not applicable to this article as no datasets were generated or analyzed in this
 %% study.

\end{document}